\documentclass{amsart}
\usepackage{amssymb,amsfonts}
\usepackage{latexsym}
\usepackage{amsmath,mathrsfs}
\usepackage[all,arc]{xy}
\usepackage{enumerate}
\usepackage[english]{babel}
\usepackage[bookmarks=true]{hyperref}
\usepackage{tikz}
\usepackage{blkarray}
\newtheorem{theo}{Theorem}[section]

\newtheorem{prop}[theo]{Proposition}
\newtheorem{lemm}[theo]{Lemma}

\theoremstyle{definition}
\newtheorem{defi}[theo]{Definition}

\newtheorem{exam}[theo]{Example}

\theoremstyle{remark}
\newtheorem{rema}[theo]{Remark}

\newcommand{\bb}[1]{\mathbb{#1}}
\newcommand{\al}[1]{\mathcal{#1}}
\newcommand{\sr}[1]{\mathscr{#1}}
\newcommand{\ak}[1]{\mathfrak{#1}}

\newcommand{\ol}{\overline}

\newcommand{\ra}{\rightarrow}
\newcommand{\lra}{\longrightarrow}
\newcommand{\x}[1]{\text{#1}}
\newcommand{\xb}[1]{\text{\textbf{#1}}}

\begin{document}
\title{Moduli spaces of anti-invariant vector bundles over a curve}
\author{Hacen ZELACI}
\address{Laboratoire de Math\'ematiques J.A. Dieudonn\'e.}
\curraddr{}
\email{z.hacen@gmail.com}
\date{\today}

\begin{abstract}
	Let $X$ be a smooth irreducible projective curve with an involution $\sigma$. A vector bundle $E$ over $X$ is called anti-invariant if there exists an isomorphism $\sigma^*E\ra E^*$. In this paper, we give a construction of the moduli spaces of anti-invariant vector bundles over $X$.
\end{abstract}
\maketitle
\tableofcontents

\section{Introduction}
Let $X$ be a smooth irreducible projective curve over $\bb C$ of genus $g_X\geqslant2$, with an involution $\sigma$. We assume that the fixed locus of $\sigma$ contains $2n$ points. Let $\pi:X\ra Y=X/\sigma$ be the associated double cover and denote by $R$ the ramification locus. \\ A vector bundle $E$ over $X$ is called anti-invariant if there exists an isomorphism $\psi:\sigma^*E\ra E^*$. If $E$ is stable, then such isomorphism is unique up to scalar, and it verifies $\sigma^*\psi=\pm\,^t\psi$. In the $"+"$ case the vector bundle  $E$ is called $\sigma-$symmetric, and in the $"-"$ case it is called $\sigma-$alternating.\\
Denote by $\al{U}_X^{\sigma,+}(r)$ the locus of stable $\sigma-$symmetric vector bundles of rank $r$ and by $\al{U}_X^{\sigma,-}(r)$ the locus of stable $\sigma-$alternating vector bundles of rank $r$. Also we denote by $\al{SU}_X^{\sigma,\pm}(r)$ the loci of $\sigma-$symmetric and $\sigma-$alternating vector bundles with trivial determinant.

In \cite{Z}, we studied the Hitchin systems over these loci and deduced a classification of their connected components. More precisely, we showed the following:
\begin{itemize}
	\item If $\pi$ is ramified, then 
	\begin{itemize}
		\item  $\al{U}_X^{\sigma,+}(r)$ and $\al{SU}_X^{\sigma,+}(r)$ are connected. 
		\item  $\al{U}_X^{\sigma,-}(r)$ has two connected components, when $r$ is even (and empty otherwise).
		\item $\al{SU}_X^{\sigma,-}(r)$ has $2^{2n-1}$ connected components (when $r$ is even).
	\end{itemize}
	\item If $\pi$ is unramified, then 
	\begin{itemize}
		\item $\al{U}_X^{\sigma,+}(r)\cong \al{U}_X^{\sigma,-}(r)$ and each one has two connected components.
		\item $\al{SU}_X^{\sigma,+}(r)$ is connected.
		\item  $\al{SU}_X^{\sigma,-}(r)$ is connected if $r$ is even, and empty otherwise. \\
	\end{itemize}
\end{itemize}   
The $2^{2n-1}$ connected components of  $\al{SU}_X^{\sigma,-}(r)$ (for even rank $r$) are  indexed by some  types $\tau=(\tau_p)_{p\in R}\mod \pm 1$, where $\tau_p\in \{\pm1\}$ is the \emph{Pfaffian} of the $\sigma-$alternating isomorphism $\psi:\sigma^*E\ra E^*$ over $p$ (the definition of the types is independent from $(E,\psi)$).\\

One can show that the stack of these anti-invariant vector bundles can be identified with the stack of parahoric $\al G-$torsors over the quotient curve $Y$, for some \emph{twisted} parahoric group schemes $\al G$ (see \cite{PR}, \cite{He}). In the first section, we elaborate this in details. \\

The main topic of this paper is to give a construction of the moduli spaces that parameterize the $\sigma-$symmetric and $\sigma-$alternating vector bundles over $X$. \\

I am very grateful to Christian PAULY for his continuous support and suggestions. \\

\section{Bruhat-Tits parahoric $\al G-$torsors}

Consider a $\sigma$-group scheme $\al G$ over $ X$, which means a smooth affine group scheme over $ X$ with an action of $\sigma$ lifted from its action on $ X$. Let $\al H=\x{Res}_{ X/Y}(\al G)^\sigma$ the invariant subgroup scheme of the Weil restriction of $\al G$ with respect to $\pi: X\ra Y$, i.e. the scheme that represents the functor $\pi_*(\al G)^\sigma$.   By \cite{BLR},  Theorem 4 and Proposition 6 the functor  $\pi_*(\al G)$ is representable, hence the subfunctor $\pi_*(\al G)^\sigma$ is representable too. We assume that $\al H$ is not empty.

\begin{defi}[$(\sigma, \al G)-$bundle] \label{sigmagbundle} A $(\sigma, \al G)-$bundle over $ X$ is a $\al G-$bundle $E$ over $ X$ with a lifting of the action of $\sigma:X\ra X$ to the total space of $E$  such that for each $x\in E$ and $g\in\al G$, we have $\sigma(x\cdot g)=\sigma(x)\cdot\sigma(g)$.
\end{defi}

Note that the action of $\sigma$ on $E$ is not a $\al G-$morphism. But it gives an isomorphism of total spaces (by the universal property of the fiber product) $$\xymatrix{E\times_{ X} X\ar[rr]^\varphi\ar[rd]&&E\ar[ld]\\& X}$$ 
which verifies $$\varphi(v\cdot g)=\varphi(v)\cdot \sigma(g)$$ 
for $g\in \al G\times_{ X} X$ and $v\in E\times_{ X} X$. This is again not a $\al G-$morphism, but we can associate to it canonically a $\al G$-isomorphism (over the identity of $ X$) $$E \stackrel{\sim}{\longrightarrow} E^\sigma$$
where $E^\sigma=(E\times_{X} X)\times^{\al G}\al G$, here $\al G$ acts on itself via $\sigma$.
\begin{theo}{\cite{BS}}\label{BStheo}
	Let $\sr M_{X}(\sigma,\al G)$ be the moduli stack of $(\sigma,\al G)-$bundles over $ X$, then we have an isomorphism $$\sr M_{X}(\sigma,\al G)\stackrel{\sim}{\longrightarrow} \sr M_Y(\al H)$$
	given by the invariant direct image $\pi^\sigma_*$.  
\end{theo}

We return now to our situation. Let $\al G=\x{Res}_{ X/Y}(\xb{SL}_r)^{\sigma}$,   where $\xb{SL}_r$ is the constant group scheme $ X\times \text{SL}_r$  over $X$ and the action of $\sigma$ on $\xb{SL}_r$ is given  by $$\sigma(x,g)=(\sigma(x),\, ^tg^{-1}).$$ Fix  a $\sigma-$alternating vector bundle $(F_\tau,\psi_\tau)$ of type $\tau$. 
Define   $\al P_\tau=\x{Aut}({F}_\tau)$,  it is a group scheme \'etale locally isomorphic to $\xb{SL}_r$. The isomorphism $\psi_\tau:\sigma^*F_\tau\ra F_\tau^*$ induces an involution, denoted $\sigma^\tau$, on $\al P_\tau$ given by  $$\alpha\lra \,^t\psi^{-1}_\tau\circ \sigma^*(\,^t\alpha^{-1})\circ \, ^t\psi_\tau,$$ so $(\sigma^\tau,\al P_\tau)$ is $\sigma-$group scheme over $X$.  Finally define  the group scheme $$\al H_\tau=\x{Res}_{X/Y}(\al P_\tau)^{\sigma^\tau}.$$
We remark that $\al G$ and $\al H_\tau$ are nonempty.

\begin{prop}\label{parahoric}
	The group schemes $\al G$ and $\al H_\tau$ are smooth affine separated group schemes of finite type which are parahoric Bruhat-Tits group schemes.  If $r\geqslant 3$, $\al G$ and $\al H_\tau$ are not generically constant.  The set of $y\in Y$ such that $\al G_y$ and $(\al H_\tau)_y$ are not semi-simple is exactly the set of branch points of the double cover $\pi: X\ra Y$.
\end{prop}
\begin{proof}
	For the first part, we refer to \cite{BLR}, Section 7.6, proposition 5.  As well as \cite{EB}, Proposition 3.5. Moreover, by \cite{PR2}, \S4, taking $I=\{0\}$, we deduce that  $\al G(\sr O_p)$ is a parahoric subgroup of $\al G(\sr K_p)$, where here $\sr O_p$ is the completion of the local ring at the branch point $p\in Y$, and $\sr K_p$ its fraction field. Further, for every $p\in B$, one can show that the flag variety $\al G(\sr K_p)/\al G(\sr O_p)$ (resp. $\al H_\tau(\sr K_p)/\al H_\tau(\sr O_p)$) is a direct limit of symplectic  (resp. special orthogonal) Grassmannian which is proper, hence these flag varieties are ind-proper. But by \cite{Ku2}, the group $\al G$ is parahoric if and only if its associated flag variety is ind-proper. So we deduce that $\al G(\sr O_p)$ (resp. $\al H_\tau(\sr O_p)$) are parahoric subgroups of $\al G(\sr K_p)$ (resp. $\al H_\tau(\sr K_p)$).  \\ 
	
	We can calculate the fibers of $\al G$ explicitly, let $x\in X$ not in $R$ (recall that $R$ is the divisor of ramification points), let  $y=\pi(x)$. 
	By definition, we have $$\al G_y=\xb{SL}_r(\pi^{-1}(y))^\sigma=(\x{SL}_r\times \x{SL}_r)^{\sigma},$$ where $\sigma(g,h)=(^th^{-1},^tg^{-1})$. So $$\al G_y=\{(g,\,\,^tg^{-1})\;|\;g\in \x{SL}_r\}\cong \text{SL}_r.$$
	Now, take $p\in B$, $\pi^{-1}(p)$ is, scheme theoretically, a double point $\x{Spec}(\bb C[\varepsilon])$, with $\varepsilon^2=0$, this gives $$\al G_p=\xb{SL}_r(\pi^{-1}(p))^\sigma=\xb{SL}_r(\bb C[\varepsilon])^\sigma,$$  
	where the action of $\sigma$ on $\bb C[\varepsilon]$ is given by $\varepsilon\ra -\varepsilon$. So $\al G_p$ is the group of elements $(g,h)$ such that 
	$$\;\;g+\varepsilon h= \,^t(g-\varepsilon h)^{-1}= (^tg-\varepsilon \,^th)^{-1}$$
	$$=\;^tg^{-1}+\varepsilon \, ^tg^{-1}\,^th^tg^{-1},$$
	and $$\x{det}(g+\varepsilon h)=1.$$
	In other words $g=\,^tg^{-1}$, $^tgh=\,^t(^tgh)$ and $g+\varepsilon h$ has determinant $1$. So $g\in \text{SO}_r(\bb C)$, and $h$ is an $r\times r$ matrix such that  $^tgh$ is symmetric.  The last condition on the determinant  is equivalent to $$\x{det}(I_r+\varepsilon \;^tgh)=1+\varepsilon\x{Tr}(^tgh)=1.$$
	Hence $\x{Tr}(^tgh)=0$. 
	It follows that $\al G_p$ is isomorphic to  $\x{SO}_r(\bb C)\times \x{Sym}^0_r(\bb C)$ with low given by $(g,h)(k,l)=(gk,gl+hk)$, where $\x{Sym}_r^0(\bb C)$ is the additive group of symmetric traceless matrices.
	We have a non split exact sequence: $$0\ra \x{Sym}_r^0(\bb C)\ra\al G_p\ra \text{SO}_r(\bb C)\ra 1.$$   Note that $\al G_p$ is not semi-simple. 
	
	With the same calculation, assume that $r$ is even, we get   $$(\al H_\tau)_p\cong \text{SL}_{r}\;\, \text{for } p\in Y \x{ not a branch point},$$ 
	and for a branch point $p$ we have  $$0\ra ASym_{r,p}^0\ra (\al H_\tau)_p \ra \text{Sp}_{r}(\bb C)\ra 0,$$
	where  $$ASym_{r,p}=\{h\in M_r| Tr(h)=0,\, M_ph=\,^thM_p=-\,^t(M_ph)\},$$
	where $M_p=(\,^t\psi_\tau^{-1})_p$ and $\text{Sp}_{r}(\bb C)$ is the symplectic group over $\bb C$.
\end{proof}

Let $\sr{SU}_X^{\sigma,+}(r)$ (resp. $\sr{SU}_X^{\sigma,\tau}(r)$) the stack defined by associating to a $\bb C-$algebra $R$ the groupoid of $(E,\delta,\psi)$, where $E$ is a $\sigma-$symmetric (resp. $\sigma-$ alternating of type $\tau$) vector bundle over $X_R=X\times \x{Spec}(R)$, $\delta$ a trivialization of $\x{det}(E)$ and a $\sigma-$symmetric (resp. $\sigma-$alternating of type $\tau$) isomorphism $\psi:\sigma^*E\stackrel{\sim}{\lra} E^*$ which is compatible (in the obvious sens) with $\delta$. 

\begin{prop} \label{parabundles}Let $\sr M_Y(\al G)$ (resp. $\sr M_Y(\al H_\tau)$) be the stack of right $\al G-$torsors (resp.  $\al H_\tau-$torsors) on $Y$, then $\sr M_Y(\al G)$ (resp. $\sr M_Y(\al H_\tau)$) is a smooth algebraic stack, locally of finite type, which is isomorphic to $\sr{SU}^{\sigma,+}_X(r)$ (resp. $\sr{SU}^{\sigma,\tau}_X(r)$).
\end{prop}
\begin{proof}
	The first part of the theorem is proved in \cite{He}, Proposition 1.\\
	By Theorem \ref{BStheo}, $\sr M_Y(\al G)\cong \sr M_X(\sigma,\xb{SL}_r)$. So it is sufficient to show $\sr M_X(\sigma,\xb{SL}_r)\cong\sr{SU}^{\sigma,+}_X(r)$. 
	Let $R$ be a $\bb C-$algebra, and $(E,\delta,\psi)$ be an element of $\sr M_X(\sigma,\xb{SL}_r)(R)$. 
	Consider  the  automorphism of the frame bundle $\tilde{E}:=\x{Isom}(\al O_{X_R}^{\oplus r},E)$ given by $$ \tilde{\psi}(f)=\,^t(\psi\circ\sigma^*f)^{-1},$$ for $f\in \tilde{E}$ (we identify $\sigma^*(\al O_{X_R}^{\oplus r}) \cong\al O_{X_R}^{\oplus r}$ using the canonical linearization). Since $\sigma^*\psi=\,^t\psi$, we get $\tilde{\psi}\circ\tilde{\psi}(f)=f,$
	thus $$\tilde{\psi}^2=id,$$ so $\tilde{\psi}$ is a lifting of the action of $\sigma$ to $\tilde{E}$, and any other lifting differs by an involution of $\al O_{X_R}^{\oplus r}$. Moreover, for $g\in \xb{SL}_r(R)$, we have $$\tilde{\psi}(f\cdot g)=\tilde{\psi}(f)\cdot\sigma(g),$$
	where $\sigma(g)=\, ^tg^{-1}$. Thus $\tilde{E}$ is $(\sigma, \xb{SL}_r)-$bundle.
	
	Conversely,   a $\al G-$bundle $E$ over $Y_R$ gives, by theorem \ref{BStheo},  a $(\sigma,\xb{SL}_r)-$bundle over $X_R$ denoted again by $E$. Let $\tilde{\psi}$ the action of $\sigma$ on $E$. Then  $$E(\bb C^r):=E\times^{\xb{SL}_r}\bb C^r$$ is $\sigma-$anti-invariant vector bundle. Let $U$ be a $\sigma-$invariant open subset of $X_R$ such that $E(\bb C^r)|_U$ is trivial, fix a $\sigma-$invariant trivialization $\varphi:\al O_U^{\oplus r}\ra E(\bb C)|_U$. Define  $\psi|_U=\,^t\tilde{\psi}(\varphi)^{-1}\circ \sigma^*\varphi^{-1}$, then $\psi$ is a $\sigma-$symmetric isomorphism $\sigma^*E(\bb C^r)|_U\ra E(\bb C^r)^*$. Gluing such local isomorphisms, we get an isomorphism $\psi:\sigma^*E(\bb C^r)\ra E(\bb C^r)^*$. Hence we get an element of $\sr M_X(\sigma,\xb{SL}_r)(R)$.
	
	Now, let $(E,\psi)$ be a $\sigma-$alternating vector bundle. Consider the bundle $$\tilde{E}=\x{Isom}(F_\tau,E).$$ It is an $\al H_\tau-$bundle. Moreover, $\psi$ induces an automorphism $\tilde{\psi}$ on $\tilde{E}$ given by $$\tilde{\psi}(f)= \,^t\psi^{-1}\circ \,^t(\sigma^*f)^{-1}\circ\,^t\psi_\tau. $$ Clearly this is an involution which makes $\tilde{E}$ a $(\sigma^\tau,\al P_\tau)-$bundle. 
	
	Conversely, a $(\sigma^\tau,\al P_\tau)-$ bundle gives, with exactly the same method as before, a $\sigma-$alternating vector bundle.
\end{proof}

\section{ $\sigma-$quadratic and $\sigma-$alternating modules} \label{app}
In this section, we will introduce the notion of $\sigma-$quadratic and $\sigma-$alternating  modules and study their moduli. This will be used later in the construction of the moduli space of $\sigma-$symmetric and $\sigma-$alternating anti-invariant vector bundles. Mainly we will elaborate only  the case of $\sigma-$quadratic modules. Our main reference here is \cite{So}.\\

Let $W$ be a finite dimension vector space with an involution $\sigma$, and $H$ a vector space. A $\sigma-$quadratic (resp. $\sigma-$alternating) form  is a linear map $q:H\lra H^*\otimes W$ such that for all $x,y\in H$ $$q(x)(y)=\sigma(q(y)(x)) \;\;(\x{resp. } q(x)(y)=-\sigma(q(y)(x))).$$
A $\sigma-$quadratic (resp. $\sigma-$alternating) module with values in $W$ is a pair $(H,q)$ as above. A map between two $\sigma-$quadratic or $\sigma-$alternating  modules $(H,q)$ and $(H',q')$ is a linear map $f:H\ra H'$ such that $$q=(\,^tf\otimes\x{id})\circ q'\circ f.$$ For a vector subspace $V\subset H$, we define its orthogonal to be $$V^{\perp_\sigma}=\{x\in H\,|q(x,y)=0\;\forall y\in V\}.$$ A $\sigma-$isotropic (resp. totally $\sigma-$isotropic) subspace $V$ of $(H,q)$ is a vector subspace such that $V\cap V^{\perp_\sigma}\not =0$ (resp. $V\subset V^{\perp_\sigma}$). We will mainly use the notion of totally $\sigma-$isotropic as we will see later on. Also, from now on, we consider only the $\sigma-$quadratic modules. Similar results about semistability, filtrations and $S-$equivalence of $\sigma-$alternating forms can be checked in this case too. We omit the details.

\begin{defi}
	The $\sigma-$quadratic module $(H,q)$ is called semi-stable (resp. stable) if for any non-zero totally $\sigma-$isotropic vector subspace $V\subset H$ we have $$\x{dim}(V)+\x{dim}(V^{\perp_\sigma})\leqslant \x{dim}(H)\;\;(\x{resp.}\;<).$$
\end{defi}
Remark that a semi-stable $\sigma-$quadratic module is necessarily injective.

Denote by $\Gamma(H,W)^\sigma$ the vector space of $\sigma-$quadratic forms $q:H\ra H^*\otimes W$, and let $P(H,W)^\sigma=\bb P\Gamma(H,W)^\sigma$. The group  $\x{SL}(H)$ acts linearly in a natural way  on $\Gamma(H,W)^\sigma$ by associating to $q$ the $(\,^tg^{-1}\otimes\x{id})\circ q\circ g^{-1}$.  This action induces clearly an action on $P(H,W)^\sigma$.
\begin{prop} \label{ssqad}
	A $\sigma-$quadratic module $(H,q)$ is semi-stable (resp. stable) if and only if the point $[q]\in P(H,W)^\sigma$ is semi-stable (resp. stable) with respect to the action of $\x{SL}(H)$.
\end{prop}
\begin{proof}We use Hilbert-Mumford criterion (\cite{Po} Theorem $6.5.5$) and we use also their notation for the weight.
	Assume that $q$ is semi-stable $\sigma-$quadratic form on $H$, let $\lambda$ be a non trivial one parameter subgroup of $\x{SL}(H)$. Consider the eigenvalue decomposition of $$H=\bigoplus_{i=1}^sH_i,$$ where the restriction of  $\lambda(t)$ to $H_i$ equals $t^{-m_i}\x{id}$, we assume also that $m_1<\cdots <m_s$. Since $\lambda(t)\in \x{SL}(H)$, we have $$\sum_{i=1}^sm_i\x{dim}(H_i)=0.$$ 
	Note that since $\lambda$ is not trivial, there exists $k$  such that $m_{k}<0 \leqslant m_{k+1}$. 
	Now $q$ decomposes  as  $q=\left(q_{ij}\right)_{ij}$, where $q_{ij}:H_i\lra H_j\otimes W$. It follows that the Hilbert-Mumford weight of $q$ is equal to $$\mu(\lambda,q)=-\x{min}\{m_i+m_j\;|\;\forall (i,j)\;\x{ such that }q_{ij}\not=0\}.$$ Suppose that  $\mu(\lambda,q)<0$ and let $V=\oplus_{i=1}^kH_i$. Then $$ V\oplus\bigoplus_{i\in I}H_i \;\subset V^{\perp_\sigma},$$ where  $I=\{i\geqslant k+1\;|\; m_j+m_i\leqslant0\;\x{ for all }j\leqslant k\}$. In particular $V$ is totally $\sigma-$isotropic. Let $l=\x{max}(I)$, so we get 
	\begin{align*}
		m_{l+1}\left(\x{dim}(V)+\x{dim}(V^{\perp_\sigma})\right)&\geqslant m_{l+1}\sum_{i=1}^k\x{dim}(H_i)+m_{l+1}\sum_{i=1}^l\x{dim}(H_i) \\ 
		&> -\sum_{i=1}^km_i\x{dim}(H_i)+m_{l+1}\sum_{i=1}^l\x{dim}(H_i) \\ 
		&= \sum_{i=k+1}^sm_i\x{dim}(H_i) +m_{l+1}\sum_{i=1}^l\x{dim}(H_i)\\ 
		&\geqslant m_{l+1}\sum_{i=1}^s\x{dim}(H_i) =m_{l+1}\x{dim}(H),
	\end{align*}   
	which contradicts the semistability of $q$, hence $\mu(\lambda,q)\geqslant0$. 
	
	Conversely, assume that for any $1-$parameter subgroup $\lambda$ we have $\mu(\lambda,q)\geqslant0$. Let $V\subset H$ be a totally $\sigma-$isotropic subspace with respect to $q$, and denote by $H_1$ a complementary subspace of $V$ in $V^{\perp_\sigma}$, and by $H_2$ a complementary subspace of $V^{\perp_\sigma}$ in $H$, so  we have $H=V\oplus H_1\oplus H_2$. Consider the integers 
	\begin{align*}
		m_1=2\x{dim}(H)-2\x{dim}(V)-\x{dim}(H_1),\\ m_2=\x{dim}(H)-2\x{dim}(V)-\x{dim}(H_1),\\ m_3=-2\x{dim}(V)-\x{dim}(H_1).
	\end{align*}  
	Then we have $m_3< m_2 < m_1$ and  $$m_1\x{dim}(V)+m_2\x{dim}(H_1)+m_3\x{dim}(H_2)=0.$$ Let's consider the $1-$parameter subgroup $\lambda$ of $\x{SL}(H)$ associated to the decomposition $H=V\oplus H_1\oplus H_2$ with characters given by the weights $m_1$, $m_2$ and $m_3$ (respecting the order of the decomposition). It follows that $\lambda$ acts on $q$ by the matrix $$\begin{pmatrix} 0 & 0 & t^{-m_1-m_3} \\ 0 & t^{-2m_2} & t^{-m_2-m_3} \\ t^{-m_1-m_3} & t^{-m_2-m_3} &t^{-2m_3} \end{pmatrix}.$$ By definition, we deduce that $$\mu(\lambda,q)=-\x{min}\{-2m_2,-m_1-m_3\}=2m_2,$$  and by hypothesis we have $\mu(\lambda,q)\geqslant 0$. Hence $m_2\geqslant 0$, which is exactly $$\x{dim}(V)+\x{dim}(V^{\perp_\sigma})\leqslant \x{dim}(H).$$  
\end{proof}

Let $(H,q)$ be a semi-stable and non-stable $\sigma-$quadratic module, there exists a minimal totally $\sigma-$isotropic subspace $ H_1$ of $H$  such that $\x{dim}(H_1)+\x{dim} (H_1^{\perp_\sigma})=\x{dim}(H)$. We repeat this procedure after replacing $H$ by $H_1^{\perp_\sigma}/H_1$ with its reduced $\sigma-$quadratic form. So we construct a filtration $$0\subset H_1\subset H_2\subset\cdots\subset H_k\subset H,$$ of totally $\sigma-$isotropic subspaces such that 
\begin{enumerate}
	\item[(i)] $H_{i}/H_{i-1}\subset H_{i-1}^{\perp_\sigma}/H_{i-1}$ are minimal totally $\sigma-$isotropic such that $$\x{dim} (H_{i}/H_{i-1})+\x{dim} (H_{i}/H_{i-1})^{\perp_\sigma}=\x{dim}(H_{i-1}^{\perp_\sigma}/H_{i-1}) .$$
	\item[(ii)] $H_{k}^{\perp_\sigma}/H_{k}$ is stable.
\end{enumerate}
For a subspace $V\subset H$, denote by $V^\vee$ the quotient $H/V^{\perp_\sigma}$. Assume that $V$ is totally isotropic with respect to $q$, then the form $q$ induces a morphism $\alpha:V\ra (V^\vee)^*\otimes W$. This induces a $\sigma-$quadratic module structure on $V\oplus V^\vee$ given by $$\begin{pmatrix} 0 & \sigma^*\alpha^\vee \\ \alpha & 0
\end{pmatrix}.$$

Given a $\sigma-$quadratic module $(H,q)$, we define the $\sigma-$quadratic graded module associated to $(H,q)$ to be $$\x{gr}(H,q)=H_{k}^{\perp_\sigma}/H_{k}\bigoplus_{i=1}^{k-1} (H_{i}/H_{i-1})\oplus (H_{i}/H_{i-1})^{\vee},$$ with the induced form. The integer $k$ is called the length of the graded $\sigma-$quadratic module. Two $\sigma-$quadratic modules are said $S-$equivalent if they have isomorphic graded modules.
\begin{prop} \label{sequiv}
	Let $Q(H,W)^\sigma=P(H,W)^{\sigma, ss}/\!\!/\x{SL}(H)$ be the geometric quotient of the subspace of semi-stable points $P(H,W)^{\sigma,ss}$ by $\x{SL}(H)$. Then a point of $Q(H,W)^\sigma$ represents  an $S-$equivalence class of $\sigma-$quadratic modules.
\end{prop}
\begin{proof}  The proof is the same as that of \cite{So} Proposition $2.5$. We prove it in two steps:
	\begin{enumerate}
		\item First we prove $\x{gr}(H,q)$ is in the closure of the orbit of $q$ by showing that there exists a $1-$parameter subgroup $\lambda$ of $\x{SL}(H)$ such that $\x{gr}(H,q)=\lim_{t\ra 0}\lambda(t)\cdot q$. We prove this by induction on $k$. If $k=0$, that's $(H,q)$ is stable, there is nothing to prove. Assume the result for $k-1$. Let $(H,q)$ be a semi-stable $\sigma-$quadratic module with a graded module of length $k$. Choose a \emph{minimal} totally $\sigma-$isotropic subspace $H_1\subset H_1^{\perp_\sigma}\subset H$. Let $H_2$ and $H_3$  be (any) complements of $H_1$ in $H_1^{\perp_\sigma}$ and $H_1^{\perp_\sigma}$ in $H$ respectively. Then we have the following decomposition of $q$ $$\begin{blockarray}{cccc}
		& H_1 & H_2 & H_3 \\
		\begin{block}{c(ccc)}
		H_1^*\otimes W & 0& 0 & \alpha \\
		H_2^*\otimes W & 0 & q' & \beta \\
		H_3^*\otimes W & \sigma^*\alpha^\vee & \sigma^*\beta^\vee & \gamma \\
		\end{block}
		\end{blockarray}\;,$$ 
		for some $\sigma-$quadratic module $q'$ on $H_2$ and some maps $\alpha, \beta$ and $\gamma$ (this last verifies  $\sigma^*\gamma^\vee=\gamma$, note also that $H_3\cong H_1^\vee$). Clearly the graded module associated to $q'$ is of length $k-1$ and we can apply the induction hypothesis  to obtain a $1-$parameter subgroup $\lambda'$ of $\x{SL}(H_2)$ such that $\lim_{t\ra 0}\lambda'(t)\cdot q'=\x{gr}(q')$.  Finally define $\lambda$ to be the $1-$ parameter subgroup of $\x{SL}(H)$ given by $$t\lra \begin{blockarray}{cccc}
		& H_1 & H_2 & H_3 \\
		\begin{block}{c(ccc)}
		H_1 & t & 0 & 0 \\
		H_2 & 0 & \lambda' & 0 \\
		H_3 & 0 & 0 & t^{-1} \\
		\end{block}
		\end{blockarray}\;.$$ 
		We see immediately that $\lim_{t\ra 0}\lambda(t)\cdot q=\x{gr}(H,q)$.  
		\item  We show here that the orbit of a $\sigma-$quadratic graded module $\x{gr}(H,q)$ is closed. Again we use induction on the length $k$. If $k=0$, then $q$ is stable. For every $1-$parameter subgroup of $\x{SL}(H)$, let $q_0=\lim_{t\ra 0}\lambda(t)\cdot q$. Since $q$ is stable, its orbit is proper. So by the valuative criterion of properness, we deduce that $q_0$ is in the orbit of $q$. Assume now the result for $k-1$,  let $\lambda$ be a $1-$parameter subgroup and assume that the limit $q_0=\lim_{t\ra 0}\lambda(t)\cdot q$ exists. Let $\x{gr}(H,q)=H_1\oplus H_1^\vee\oplus H'$. So the $\sigma-$quadratic form on $\x{gr}(H,q)$ can be written   $$\begin{blockarray}{cccc}
		& H_1 & H' & H_1^\vee \\
		\begin{block}{c(ccc)}
		H_1^*\otimes W & 0& 0& \alpha\\
		H'^*\otimes W& 0  & q'  & \beta \\
		(H_1^\vee)^*\otimes W & \sigma^*\alpha^\vee & \sigma^*\beta^\vee& \gamma \\
		\end{block}
		\end{blockarray}\;.$$ 
		Denote $H_1(t)=\lambda(t)(H_1)$,  $\alpha_t=\lambda(t)\cdot\alpha$, $\beta_t=\lambda(t)\cdot \beta$ and $\gamma_t=\lambda(t)\cdot\gamma$. The subspace $H_1(t)$ is totally $\sigma-$isotropic with respect to $q_t=\lambda(t)\cdot q$ and the module $H_1\ra (H_1^\vee)^*\otimes W$ is stable.  We can assume that $\lambda(t)$ (for all $t\in \bb C^*$) stabilizes $H_1$ and $H_1^{\perp_\sigma}=H_1\oplus H'$. Hence we can write  $\lambda(t)^{-1}$ in the form
		
		$$\begin{blockarray}{cccc}
		& H_1 & H' & H_1^\vee \\
		\begin{block}{c(ccc)}
		H_1 & f(t)& g(t)& h(t)\\
		H'& 0  & u(t)  & v(t) \\
		H_1^\vee & 0 & 0 & w(t) \\
		\end{block}
		\end{blockarray}\;.$$ 
		Moreover, without changing $q_t$, we can assume that $$\x{det}(f(t))=\x{det}(u(t))=\x{det}(w(t))=1.$$
		It follows that $\alpha_t=\,^tf(t)\alpha w(t)$. Since $\alpha$ is stable, and since $\alpha_t$ has a limit by assumption, it follows, by the valuative criterion of properness, that $f(t)$ and $w(t)$ have limits $f_0$ and $w_0$ in $\x{SL}(H_1)$ and $\x{SL}(H_1^\vee)$ respectively. Moreover, By the induction hypothesis, we deduce that $u(t)$ has a limit $u_0$ in $\x{SL}(H')$. Now, up to multiplication with the diagonal matrix whose entries are $f(t)$, $u(t)$ and $w(t)$, we can assume that entries of the diagonal of $\lambda(t)^{-1}$ are $1$.  Using that, we can explicitly calculate $\beta_t$ and $\gamma_t$ in function of $g(t)$, $h(t)$ and $v(t)$. Indeed we have $$\beta_t=q'\circ v(t)+g(t)\circ \alpha+\beta,$$ $$\gamma_t=\sigma^*\alpha^\vee\circ h(t)+v(t)\circ q'\circ v(t)+\sigma^*\beta^\vee\circ v(t)+h(t),$$ using again the properness, we deduce the existence of limits of $g(t)$, $h(t)$ and $v(t)$. This ends the proof.
		
	\end{enumerate}
\end{proof}

\section{Moduli spaces of anti-invariant vector bundles}
In this section, we construct the moduli space of $\sigma-$symmetric anti-invariant vector bundles. 

\subsection{Semistability of anti-invariant vector bundles}
Let $(E,\psi)$ be an anti-invariant vector bundle over $X$. We say that a subbundle $F$ of $E$ is $\sigma-$isotropic if the induced map $\psi:\sigma^*F\ra F^*$ is identically zero. 
\begin{defi} Let $(E,\psi)$ be an anti-invariant vector bundle over $ X$. We say that it is \emph{semi-stable} (resp. \emph{stable}) if for every $\sigma-$isotropic sub-bundle $F$ of $E$, one has $$\mu(F)\leqslant 0 \;(\text{resp. } \mu(F)<0).$$
\end{defi}

\begin{prop}\label{ssequiv}  $(E,\psi)$ is semi-stable if and only if $E$ is semi-stable vector bundle.
\end{prop}
\begin{proof}
	We follow the same lines of the proof of \cite{R} $4.2$, page $155$. 
	
	The "if " part is obvious.
	Conversely, take $F$ to be any sub-bundle of $E$. Define $F^{\perp_\sigma}$ to be the kernel of the surjective morphism:
	$$E\stackrel{\sim}{\longrightarrow}\sigma^* E^*\twoheadrightarrow \sigma^* F^*.$$
	Note that $F^{\perp_\sigma}$ have the same degree as $F$, and $F$ is $\sigma-$isotropic if and only if $F\subset F^{\perp_\sigma}$.\\  
	Then, the sub-bundle $N$ of $E$ generated by $F\cap F^{\perp_\sigma}$ is $\sigma-$isotropic. Indeed, we have $N\subset F$, so $F^{\perp_\sigma}\subset N^{\perp_\sigma}$, interchanging $F$ and $F^{\perp_\sigma}$ we get $F\subset N^{\perp_\sigma}$, hence $N\subset N^{\perp_\sigma}$. Let $M$ be the image of $F\oplus F^{\perp_\sigma}$ in $E$. We have  $M=N^{\perp_\sigma}$, to see this, note that $N^{\perp_\sigma}$ contains $F$ and $F^{\perp_\sigma}$, so it contains $M$, but this two bundles have the same rank.   Moreover we  have $$0\ra N\ra F\oplus F^{\perp_\sigma}\ra M\ra 0,$$ which implies also 
	$$0\ra M^{\perp_\sigma}\ra F\oplus F^{\perp_\sigma}\ra N^{\perp_\sigma}\ra 0,$$ we deduce that they have the same degree too. Hence they are equal.\\
	Therefore, $\text{deg}(N)=\text{deg}(F)$, but $\text{deg}(N)\leq 0$ because it is $\sigma-$isotropic and $(E,\psi)$ is semi-stable by hypothesis, so $E$ is semi-stable as a vector bundle.
\end{proof}

Let $E$ be a semistable $\sigma-$symmetric anti-invariant vector bundle, the following Lemma generalizes the isotropic filtration of self-dual vector bundle. 
\begin{lemm}
	There exists a filtration of $E$ of the form 
	$$0=F_0\subset F_1\subset\cdots\subset F_k\subseteq F_k^{\perp_\sigma}\subset F_{k-1}^{\perp_\sigma}\subset\cdots\subset F_0^{\perp_\sigma}=E,$$
	where $F_i$ are degree $0$ sub-bundles of $E$ (which are of course $\sigma-$isotropic) such that $F_i/F_{i-1}$ is stable vector bundle of rank $\geqslant 1$ for $i=1,\dots,k$. 
\end{lemm} 
\begin{proof} The proof is similar to that of  Lemma $1.9$ of \cite{GH}.\\ 
	The proof is a constructive one, we consider the set of all $\sigma-$isotropic subbundles of $E$, which contains $0$ and $E$. If $E$ is stable anti-invariant vector bundle, then it has no $\sigma-$isotropic proper sub-bundle of degree $0$, and the filtration is $0\subset0^{\perp_\sigma}=E$. Otherwise, let $F_1$ be  a $\sigma-$isotropic sub-bundle of $E$ of degree $0$ and smallest rank (it is a stable vector bundle, because otherwise, a proper sub-bundle of $F_1$ of degree $0$ would be a $\sigma-$isotropic sub-bundle of $E$, contradicting the minimality of $\x{rk}(F_1)$). Now, we repeat this procedure on $E/F_1$ instead of $E$.
\end{proof}
\begin{lemm} Consider the above filtration, then we have $$\sigma^* (F_{i-1}^{\perp_\sigma}/F_i^{\perp_\sigma})\cong\left(F_i/F_{i-1}\right)^*,\;\,\sigma^*(F_k^{\perp_\sigma}/F_k)\cong (F_{k}^{\perp_\sigma}/F_k)^*,$$
	for $i=1,\dots,k$.
\end{lemm}
\begin{proof}
	For $i=1$, this is just the definition of $F_1^{\perp_\sigma}$. Let $i>1$, and consider $$0\subset F_{i-1}\subset F_{i}\subset F_{i}^{\perp_\sigma}\subset F_{i-1}^{\perp_\sigma}\subset E.$$
	We have a commutative diagram $$\xymatrix{0\ar[r]& F_{i-1}^{\perp_\sigma}\ar[r]^i& E\ar[r]\ar[rd]^{p_1}& \sigma^*F_{i-1}^*\ar[r]&0 \\ & & & \sigma^*F_{i}^*\ar[u]^{p_2}& \\ & & &  \sigma^*(F_i/F_{i-1})^*\ar@{^{(}->}[u]&.}$$ Since the composition $p_2\circ p_1\circ i$ is identically zero, it follows that $p_1\circ i$ factorizes through $\sigma^*(F_i/F_{i-1})^*$. The resulting map $F_{i-1}^{\perp_\sigma}\ra\sigma^*(F_i/F_{i-1})^*$ is nonzero map because otherwise $F_{i-1}^{\perp_\sigma}\subset F_{i}^{\perp_\sigma}$, thus $F_i/F_{i-1}=0$ which contradicts the definition of the above filtration. Its kernel contains $F_{i}^{\perp_\sigma}$, so we obtain a nonzero map $$F_{i-1}^{\perp_\sigma}/F_{i}^{\perp_\sigma}\ra\sigma^*(F_i/F_{i-1})^*.$$
	But this two bundles are stable of the same rank and degree, so the last map has to be an isomorphism. \\ 
	For $i=k$, we have a nonzero map $$F_{k}^{\perp_\sigma}/F_{k}\ra\sigma^*(F_k^{\perp_\sigma}/F_{k})^*,$$ otherwise $F_k=F_{k}^{\perp_\sigma}$. So the same argument as before gives the result.
\end{proof}
The above lemma proves that the bundle $$\x{gr}^\sigma(E)=\bigoplus_{i=1}^k\left(F_i/F_{i-1}\oplus F_{i-1}^{\perp_\sigma}/F_i^{\perp_\sigma}\right)\oplus (F_{k}^{\perp_\sigma}/F_k) $$
is an anti-invariant vector bundle. Moreover, it is $\sigma-$symmetric (resp. $\sigma-$alternating) if $E$ is $\sigma-$symmetric (resp. $\sigma-$alternating). 
\begin{defi} The vector bundle $\x{gr}^\sigma(E)$ is called the \emph{$\sigma-$graded bundle} associated to $(E,\psi)$. Two $\sigma-$symmetric or $\sigma-$alternating anti-invariant vector bundles $E$ and $F$ are said to be \emph{$S-$equivalent} if their associated $\sigma-$graded bundles  $\x{gr}^\sigma(E)$ and $\x{gr}^\sigma(F)$ are isomorphic.\\  
\end{defi}
\begin{exam}
	We give an example of two non-isomorphic  $\sigma-$symmetric anti-invariant vector bundles which are $S-$equivalent. Let $M$ be an element of $\x{Prym}_{X/Y}$, and $\phi:\sigma^*M\stackrel{\sim}{\lra} M^*$. The vector bundle $M^{\oplus2}$ with the $\sigma-$symmetric isomorphism $$\psi=\begin{pmatrix}
	0 & \phi \\ \phi & 0
	\end{pmatrix}$$ is a $\sigma-$symmetric anti-invariant vector bundle. Now for $\eta\in \x{Ext}^1(M,M)_-\cong H^1(X,\al O_X)_-$, where the involution on this vector space is given by pullback by $\sigma$, consider the associated extension of $M$ by $M$ $$0\ra M\ra E\ra M\ra 0.$$ Note that in rank $2$ taking the dual does not change the extension class in $H^1(X,\al O_X)$ because of the formula $E^*\cong E\otimes \x{det}(E)^{-1}$. \\
	Since $\eta$ is a $-1$ eigenvector, $E$ is anti-invariant. Indeed,  by pulling back by $\sigma$ we get the extension $$\xymatrix{0\ar[r] &\sigma^*M\ar[r]\ar[d]^\simeq & \sigma^*E\ar[r]\ar[d]^\simeq & \sigma^*M\ar[r]\ar[d]_\simeq & 0 \\  0\ar[r] &M^{-1}\ar[r] & E\otimes M^{-2}\ar[r] & M^{-1}\ar[r] & 0.} $$  But $E\otimes M^{-2}$ is  isomorphic to $E^*$.  Moreover, if $\eta\not=0$ then $E$ is not isomorphic to $M^{\oplus 2}$. However, clearly $E$ and $M^{\oplus2}$ are $S-$equivalent as $\sigma-$symmetric anti-invariant vector bundles.\\
\end{exam}

\subsection{Construction of the moduli space}\label{constrcuction}
We follow the method of \cite{So} to construct this moduli space. \\
\subsubsection{$\sigma-$symmetric case}

Fix an ample $\sigma$-linearized line bundle $(\al O(1),\eta)$ of degree $1$ over $X$ (in the \'etale case, there are no such bundle, so one has to take degree $2$ instead of degree $1$, but this doesn't produce any difference). Let $\nu$ be some big integer such that for any semi-stable coherent sheaf $E$ over $X$ of rank $r$ and degree $0$, we have $H^1(X,E(\nu))=0$  and $E(\nu)$ is generated by global sections. Let $\sr F=\al O_X^m(-\nu)$  where $m=r\nu+r(1-g_X)$. Denote $H=\bb C^{m}$.\\

Consider the functor $$\text{Quot}^\sigma:(\text{algebraic varieties})\ra(\text{sets})$$ which associates to a variety $T$ the set of isomorphism classes of  $(E,q,\overline{\phi})$, where $E$ is coherent quotient sheaf $q:p_1^*\sr F\ra E$ over $X\times T$  flat over $T$, and $\overline{\phi}$ is class, modulo $\bb C^*$, of $\sigma-$symmetric isomorphism $\sigma^* E\cong E^*$ ($\sigma$ acts only on $X$), such that, for each $t\in T$, $E_t$ is a semi-stable, $\sigma-$symmetric and locally free  of rank $r$ and $q$ induces an isomorphism $H\ra H^0(X, E_t(\nu))$. Two triplets $(E,q,\overline{\phi})$ and $(F,p,\overline{\psi})$ are isomorphic if there exists an isomorphism $f:E\ra F$ such that $p=f\circ q$ and $\psi\circ \sigma^*f=\,^tf^{-1}\circ \phi$ (for some $\phi\in\overline{\phi}$ and $\psi\in \overline{\psi}$). 

Let $[E,q,\overline{\phi}]\in \x{Quot}^\sigma(\bb C)$, consider the diagram $$\xymatrix{ H\otimes\al  O_{X}\ar[r]^{\sigma^*q}&   \sigma^*E(\nu) \ar[r]\ar[d]^\phi & 0 \\  0\ar[r] & E^*(\nu)\ar[r]^{^tq\;\;\;\;\;}& H^*\otimes \al O_{X}(2\nu)\,.}$$
The composition $h=\,^tq\circ\phi\circ\sigma^*q$ gives, at the level of global sections, a $\sigma-$quadratic  form $H\ra H^*\otimes W$, where $W=H^0(X,\al O_{X}(2\nu))$ with an involution induced by the linearization on $\al O(1)$. Hence we get a point $\overline{h}\in P(H,W)^\sigma$. This actually defines a transformation $\ak H:\x{Quot}^\sigma\lra P(H,W)^\sigma$, where $P(H,W)^\sigma$ is seen as a functor by associating to a variety $T$ the space $P(H_T,W_T)^\sigma$, where $H_T=H^0(X\times T,\al O_{X\times T}^{\oplus m})$ and $W_T=H^0(X\times T,\al O_{X\times T}(2\nu))$.\\ 
\begin{prop} \label{ssequiv2} Let  $(E,\psi)$ be a $\sigma-$symmetric vector bundle, and $h$ its corresponding point of $\Gamma(H,W)^\sigma$, then the following are equivalent: 
	\begin{itemize}
		\item[(a)] The bundle $E$ is semi-stable.
		\item[(b)] $h$ is semi-stable with respect to the action of $\x{SL}(H)$.
	\end{itemize} 
	Moreover, $(E,\psi)$ is stable if and only if $h$ is stable.
\end{prop}
\begin{proof}
	Assume that $(E,\psi)$ is semi-stable,  let $V\subset H$ be a totally $\sigma-$isotropic. Denote by $F$ and $F'$ the subsheaves of $E$ generated by $V$ and $V^{\perp_\sigma}$ respectively. By Proposition \ref{ssequiv} the induced vector bundle is semi-stable, hence by \cite{Po} Proposition $7.1.1$, for all subsheaf $F$ of $E$, one has $$\dfrac{h^0(F(m))}{\x{rk}(F)}\leqslant \dfrac{h^0(E(m))}{\x{rk}(E)},$$ for $m\geqslant \nu$ large enough. By applying this to $F$ and $F'$, and then summing up, we deduce $$h^0(F(\nu))+h^0(F'(\nu))\leqslant h^0(E(m)),$$
	which is the same as $$\x{dim}(V)+\x{dim}(V^{\perp_\sigma})\leqslant \x{dim}(H).$$
	Hence $(H,h)$ is semi-stable. So by Proposition \ref{ssqad}, ${h}$ is semi-stable with respect to the action of $\x{SL}(H)$. 
	
	Conversely, suppose that ${h}$ is semi-stable, then by Proposition \ref{ssqad}, $(H,h)$ is also semi-stable. Let $F$ be a $\sigma-$isotropic subbundle of $E$,  $V=H^0(F(\nu))$ and  $V'=H^0(F^{\perp_\sigma}(\nu))$.  We have $V'\subset V^{\perp_\sigma}$. Indeed, we have a commutative diagram $$\xymatrix{V'\ar[r]& H^0(E(\nu))\ar[r]^\sim \ar[rd]& H^0(\sigma^*E^{*}(\nu)) \ar[d] \ar[r]& H^*\otimes W\ar[d] \\&& H^0(\sigma^*F^*(\nu)) \ar[r] & V^*\otimes W\;.}$$    Since $F$ is totally $\sigma-$isotropic, the composition $V'\ra V^*\otimes W$ is identically zero. Hence $V'\subset V^{\perp_\sigma}$. Since we have also $V\subset V'$, we deduce that $V$ is totally $\sigma-$isotropic subspace of $H$. So we get $$\x{dim}(V)+\x{dim}(V')\leqslant \x{dim}(V)+\x{dim}(V^{\perp_\sigma})\leqslant \x{dim}(H).$$
	It follows 
	\begin{align*}
		\x{dim}(V)+\x{dim}(V')& = \x{deg}(F)+ \x{rk}(F)\nu + \x{deg}(F^{\perp_\sigma})+\x{rk}(F^{\perp_\sigma})\nu +r(1-g_X)\\ 
		&= 2\x{deg}(F) +r\nu +r(1-g_X) \\ 
		&\leqslant r\nu + r(1-g_X)=\x{dim}(H).
	\end{align*}
	Hence $\x{deg}(F)\leqslant 0$. This proves that $E$ is semi-stable. 
\end{proof} 

Now, let $i\geqslant 0$ and denote by $H_i=H\otimes H^0(\al O_X(i))$, $W_i=H^0(\al O_X(2\nu+2i))$. For a $\sigma-$quadratic module $(H,h)$, we denote by $(H_i,h_i)$ the $\sigma-$quadratic module obtained as follows: taking  the tensor product with $\al O(i)$ we obtain  $$H\otimes \al O(i)\lra H^*\otimes \al O(i)\otimes W\lra H^*\otimes \al O(i)\otimes W\otimes H^0(\al O(i))^*\otimes H^0(\al O(i)).$$ Then at the level of global sections we deduce  
$$H_i\lra H_i^*\otimes  W \otimes H^0(\al O(i))^2\lra H_i\otimes W_i,$$
and the composition is denoted $h_i$. 

Let $Z\subset P(H,W)^\sigma$ be the locus of $\sigma-$quadratic forms $\overline{h}$ such that $$\x{rk}(h_i)\leqslant r(\nu+i-g_X+1),\;\;\forall \;i\geqslant 0.$$  It is clear that $Z$ contains the image of $\ak H(\bb C)$. Moreover we have the following 
\begin{theo}
	Let $\al Q^\sigma\subset Z$ be the open of semi-stable points, then $\al Q^\sigma$ represents the functor $\x{Quot}^\sigma$.
\end{theo}
\begin{proof}
	We need to prove that $\ak H$ induces an isomorphism of functor between $\x{Quot}^\sigma$ and the functor of points of $\al Q^\sigma$. The main point is to show this for the $\bb C$ valued points. By Proposition \ref{ssequiv2}, we deduce that the image of $\ak H(\bb C)$ is contained in $\al Q^\sigma(\bb C)$. Giving a point $\overline{h}\in \al Q^\sigma$, fix a representative $h$ of $\overline{h}$. Taking the tensor product with $\al O_X(-\nu)$ gives $$H\otimes \al O_X(-\nu)\stackrel{h}{\lra} H^*\otimes W\otimes \al O_X(-\nu)\stackrel{ev}{\lra} H^*\otimes \al O_X(\nu).$$
	Let $F=\x{Ker}(ev\circ h)$ and $E=H\otimes\al O_X(-\nu)/F$. $E$ doesn't depend on the chosen representative of $\overline{h}$ and we have the following commutative diagram 
	$$\xymatrix{0\ar[r] & F\ar[r] & H\otimes\al O_X(-\nu)\ar[r]\ar[d]^{ev\circ h} & E\ar[r]  &0\\ 0\ar[r]&\sigma^*E^*\ar[r]& H^*\otimes \al O_X(\nu) \ar[r]^{\;\;\;\;\;\;p}& \sigma^*F^* \ar[r] &0.  }$$ By definition, $ev\circ h$ vanishes over $F$, hence it factorizes through $E$ giving an injective map $f:E\ra H^*\otimes \al O_X(\nu)$, since $h$ is $\sigma-$symmetric, we deduce that $p\circ ev\circ h=\sigma^*(\,^th\circ\,^tev\circ \,^tp)=0$, so the map $f$ gives a $\sigma-$symmetric morphism $\psi:\sigma^*E\ra E^*$, which is clearly injective. 
	
	Let $s$ be the rank of $E$ and $d$ its degree. By what we have just said we deduce $d\leqslant0$.  From the condition defining $Z$, we deduce that for all $i$  $$d+s(\nu+i+1-g_X)\leqslant \x{rk}(q_i)\leqslant r(\nu +i+1-g_X),$$  so in particular we deduce that $r\geqslant s$. But since $q$ is semi-stable, the map $q:H\ra H^*\otimes W$ is injective, hence $H^0(F(\nu))=0$.  Thus the map $H\ra H^0(E(\nu))$ is injective and we deduce $$r(\nu +1-g_X) \leqslant d+s(\nu+1-g_X),$$ hence $d\geqslant 0$, thus $d=0$. It follows that $s\geqslant r$, and  so $r=s$. Hence $\psi$ is surjective. So $(E,\psi)$ is a $\sigma-$symmetric vector bundles.  
	
	Using the universal family over $\al{Q}^\sigma$, one can make the above construction functorial which gives an inverse to $\ak H$. 
\end{proof}

Consider the functor $$\x{Bun}_X^{\sigma,+}(r):(\x{algebraic varieties})\lra (\x{sets}),$$
that associates to  a variety $T$ the set of isomorphism classes of families $(\sr E,\psi)$ of rank $r$ $\sigma-$symmetric anti-invariant vector bundles over $X$ parameterized by $T$, such that $\sr E_t$ is semi-stable for all $t\in T$.

\begin{theo}
	Consider the good quotient $\al M_X^{\sigma,+}(r)=\x{Quot}^\sigma(\bb C)/\!\!/\x{SL(H)}$. Then $\al M_X^{\sigma,+}(r)$ is a coarse moduli space for the functor  $\x{Bun}_X^{\sigma,+}(r)$, which is a projective variety, and its underlying set consists of $S$-equivalence classes of semi-stable $\sigma-$symmetric anti-invariant vector bundles.
\end{theo}
\begin{proof}
	Consider a family $(\sr E,\psi)$ of  $\sigma-$symmetric semi-stable bundles parameterized by a variety $T$, then for $\nu$ big enough, ${p_2}_*\sr E(\nu)$ and ${p_2}_*(\sigma^*\sr E^*(\nu))$  are locally free, so by choosing local trivializations, we deduce a unique, up to an action of $\x{SL}(H)$, map to $\al Q^\sigma$ . Thus we get a morphism $T\lra \al M_X^{\sigma,+}(r)$. This is obviously functorial in $T$. 
	
	A point  $a\in \al M_X^{\sigma,+}(r)$, corresponds by  $\ak H$ to a point of $\al Q^\sigma/\!\!/\x{SL(H)}$, this transformation respects the graded $\x{gr}$. Hence, using Proposition \ref{sequiv}, we deduce that  $a$ represents an $S-$equivalence class of semi-stable $\sigma-$symmetric vector bundles. 
\end{proof}

\subsubsection{$\sigma-$alternating case} The construction of the moduli space $\al M_X^{\sigma,-}(r)$ of semi-stable $\sigma-$alternating vector bundles follows the same method as the case of  $\sigma-$symmetric vector bundles, using the moduli of $\sigma-$alternating modules rather than $\sigma-$quadratic ones. 

\subsection{Some properties of $\al M_X^{\sigma,+}(r)$ and $\al M_X^{\sigma,-}(r)$}
By Proposition \ref{ssequiv} we  have canonical forgetful maps $$\al {M}_X^{\sigma,+}(r)\ra \ol{\al{U}}_X(r,0),$$ $$\al {M}_X^{\sigma,-}(r)\ra \ol{\al{U}}_X(r,0),$$ where $\ol{\al U}_X(r,0)$ is the moduli space of semi-stable vector bundles of rank $r$ and degree $0$ over $X$.  The  images of these maps are obviously $\ol{\al{U}}^{\sigma,\pm}_X(r)$. A natural question arises: what are the degrees of these maps? 
\begin{rema}
	Note that the involution $E\ra \sigma^*E^*$ is well defined on $\ol{\al{\al U}}_X(r,0)$, since we have $\x{gr}(\sigma^*E^*)=\sigma^*(\x{gr}(E))^*$.
\end{rema}
\begin{prop} \label{forgetful1} The forgetful maps  $\al{M}_X^{\sigma,+}(r)\longrightarrow \ol{\al{U}}^{\sigma,+}_X(r)$ and $\al {M}_X^{\sigma,-}(r)\longrightarrow \ol{\al{U}}^{\sigma,-}_X(r)$  are injective. In particular they are  bijective. 
\end{prop}
\begin{proof} 
	We treat the $\sigma-$symmetric case. 	Let $(E,\psi)$ be a $\sigma-$symmetric vector bundle, suppose that $E$ is stable, so $\x{Aut}_{\x{GL}_r}(E)=\bb C^*$ and  $\psi:\sigma^*E\ra E^*$ is unique up to scalar multiplication. The action of $\x{Aut}_{\text{GL}_r}(E)$ on these $\sigma-$symmetric forms is given by $$f\cdot\psi=(\,^tf)\psi \;(\sigma^*f). $$ If $f=\xi \x{Id}_E$, with $\xi\in \bb C^*$, then this action is simply given by $\psi\ra\xi^2\psi$. It follows that this action is transitive, hence $(E,\psi)$ and $(E,\lambda\psi)$ are isomorphic as $\sigma-$symmetric vector bundles.
	
	If $E$ is strictly semi-stable $\sigma-$symmetric anti-invariant vector bundle. It is easy to see that $E$ can be decomposed as  $$E=\left(\bigoplus_{i=1}^aF_i^{\oplus f_i}\right)\oplus\left(\bigoplus_{j=1}^b G_j^{\oplus g_j}\right)\oplus\left(\bigoplus_{k=1}^c(H_k\oplus \sigma^*H_k^*)_k^{\oplus h_k}\right)$$ 
	with $F_i$, $G_j$ and $H_k$ stable vector bundles (mutually non isomorphic), such that 
	\begin{itemize}
		\item $F_i$ are  $\sigma-$symmetric (resp. $\sigma-$alternating).
		\item $G_j$ are $\sigma-$alternating  (resp. $\sigma-$symmetric).
		\item $H_k$ are not $\sigma-$anti-invariant.
	\end{itemize}
So this reduces the question to the case when  the vector bundle  $E$ is of the form $F^{\oplus d}$ or $(G\oplus \sigma^*G^*)^{\oplus d}$ for stable anti-invariant vector bundle $F$  and  stable non-anti-invariant vector bundle $G$. Now, the set of $\sigma-$symmetric isomorphisms $\psi:\sigma^*E\ra E^*$ is equal to the locus of symmetric matrices of $\x{GL}_d(\bb C)$ in both cases. Hence it is sufficient to  use the fact that a non-degenerated symmetric matrices can be decomposed in the form $^tM\times M$, for some nondegenerated matrix $M$. This shows that all the $\sigma-$symmetric isomorphisms on $E$ define the same point in the moduli space $\al{M}_X^{\sigma,+}(r)$.
	
\end{proof}



The case of vector bundles with trivial determinant is slightly different.  For simplicity we consider the forgetful maps just on the stable loci $$\al{SM}_X^{\sigma,\pm,s}(r)\lra \al{SU}_X^{\sigma,\pm}(r).$$ Here $\al{SM}_X^{\sigma,\pm,s}(r)$ is the locus of stable $\sigma-$symmetric or $\sigma-$alternating vector bundles in the moduli spaces $\al{M}_X^{\sigma,\pm}(r)$.

\begin{prop} We have two cases:
	\begin{enumerate}
		\item [$(1)$] If $r$ is odd, then the forgetful map  $\al{SM}_X^{\sigma,+}(r)\longrightarrow \al{SU}^{\sigma,+}_X(r)$ is injective.
		\item [$(2)$] If $r$ is even, the forgetful map  $\al{SM}_X^{\sigma,\pm}(r)\longrightarrow \al{SU}^{\sigma,\pm}_X(r)$ is of degree $2$.
	\end{enumerate}
\end{prop}
\begin{proof}
	Let $(E,\psi)$ be a $\sigma-$symmetric vector bundle with a trivialization of its determinant. Suppose that $E$ is stable. As $\x{Aut}_{\x{GL}_r}(E)=\bb C^*$, we see that $\x{Aut}_{\text{SL}_r}(E)=\mu_r$, where $\mu_r$ is the group of $r^{th}$ roots of unity. Remark  that the map $\mu_r\ra \mu_r$, given by $\xi\mapsto\xi^2$ is a bijection if $r$ is odd, and it is two-to-one on its image if $r$ is even.
	\begin{enumerate}
		\item [$(1)$] If $r$ is odd, since $E$ is stable, $\psi:\sigma^*E\ra E^*$ is unique up to scalar multiplication, as $\text{det}(\psi)=1$, the number of such isomorphisms is exactly $r$. The action of $\x{Aut}_{\text{SL}_r}(E)$ on these $r$ $\sigma-$symmetric forms is given by $$f\cdot\psi=(\,^tf)\psi \;(\sigma^*f).$$ As $f=\xi\x{Id}_E$, for $\xi\in \mu_r$, then  we conclude as in the proof of Proposition \ref{forgetful1} that this action is transitive.
		\item [$(2)$] Assume $r$ is even, with the same argument as above, we see that the action has two different orbits. So $E$ admits two non equivalent $\sigma-$symmetric forms. The same argument applies for the $\sigma-$alternating case.
	\end{enumerate}
\end{proof}
\begin{rema}
	Note that the above Proposition is similar to the situation of forgetful map of orthogonal bundles. See \cite{OS}.\\
\end{rema}

\bibliographystyle{alpha}
\bibliography{bib}
\end{document}